\documentclass[12pt]{amsart}
\usepackage{SSdefn}
\usepackage[lite,nobysame,alphabetic]{amsrefs}
\usepackage{verbatim}

\title{An equivariant Hilbert basis theorem}

\author{Daniel Erman}
\address{Department of Mathematics, University of Wisconsin, Madison, WI}
\email{\href{mailto:derman@math.wisc.edu}{derman@math.wisc.edu}}
\urladdr{\url{http://math.wisc.edu/~derman/}}

\author{Steven V Sam}
\address{Department of Mathematics, University of Wisconsin, Madison, WI}
\curraddr{Department of Mathematics, University of California, San Diego, CA}
\email{\href{mailto:ssam@ucsd.edu}{ssam@ucsd.edu}}
\urladdr{\url{http://math.ucsd.edu/~ssam/}}

\author{Andrew Snowden}
\address{Department of Mathematics, University of Michigan, Ann Arbor, MI}
\email{\href{mailto:asnowden@umich.edu}{asnowden@umich.edu}}
\urladdr{\url{http://www-personal.umich.edu/~asnowden/}}

\thanks{DE was partially supported by NSF DMS-1302057.}
\thanks{SS was partially supported by NSF DMS-1500069 and DMS-1651327 and a Sloan Fellowship.}
\thanks{AS was partially supported by NSF DMS-1303082 and DMS-1453893 and a Sloan Fellowship.}

\date{September 5, 2019}


\begin{document}

\begin{abstract}
We prove a version of the Hilbert basis theorem in the setting of equivariant algebraic geometry.
\end{abstract}

\maketitle

\section{Introduction}

A topological space $X$ equipped with an action of a group $G$ is {\bf $G$-noetherian} if every descending chain of $G$-stable closed subsets of $X$ stabilizes. If $X$ is a scheme equipped with an action of $G$, one says $X$ is {\bf topologically $G$-noetherian} if the topological space $\vert X \vert$ is $G$-noetherian. The notion of $G$-noetherianity has received much attention in recent years due to its connection to representation stability. We recall a few examples:
\begin{itemize}
\item Cohen \cite{cohen1,cohen2} proved that the scheme $X=\varprojlim_d \bA^d$ (or even $X^n$, for any $n \ge 0$) is topologically $S_{\infty}$-noetherian\footnote{In fact, Cohen's result is stronger and applies at the level of rings.}, and used this to prove certain results in universal algebra. This result was rediscovered some decades later by Aschenbrenner, Hillar, and Sullivant \cite{aschenbrenner,hillar}, with applications to combinatorial algebra and algebraic statistics.
\item Draisma--Eggermont \cite{draisma-eggermont} considered the scheme $X$ of $\infty \times \infty$ matrices, and showed that $X^n$ is topologically $G$-noetherian for any $n \ge 0$, with $G=\GL_{\infty} \times \GL_{\infty}$. This result was crucial to their study the equations of so-called Pl\"ucker varieties. See \cite{draisma-kuttler} for related results and applications.
\item Following work of Eggermont \cite{eggermont} and Derksen--Eggermont--Snowden \cite{DES}, Draisma \cite{draisma} proved that if $V$ is any polynomial representation of $\GL_{\infty}$ then $V^*$ is topologically $\GL_{\infty}$-noetherian. This result has been applied by the present authors \cite{ess} to prove a general finiteness result in commutative algebra.
\end{itemize}
Unfortunately, there are few general tools for proving that a space is equivariantly noetherian. The purpose of this paper is to establish one such tool: we view our main theorem as a version of the classical Hilbert basis theorem in the setting of equivariant noetherianity.

Recall that the classical Hilbert basis theorem states that if $A$ is a noetherian ring then the polynomial ring $A[x]$ is again noetherian. This can be recast in the language of schemes as follows: if $S$ is a noetherian scheme and $X \to S$ is a finite type map of schemes then $X$ is noetherian. Our main theorem is the following equivariant version of this statement.

\begin{theorem} \label{thm:Ghilbert}
Let $X \to S$ be a $G$-equivariant finite type map of schemes. Suppose $S$ is topologically $G$-noetherian. Then $X$ is topologically $G$-noetherian.
\end{theorem}

We emphasize that there are no finiteness assumptions in the above theorem except for those stated. Indeed, the theorem is most interesting for infinite dimensional schemes like those mentioned above. We note the following useful corollary:

\begin{corollary} \label{cor:Ghilbert}
Let $X$ be a scheme equipped with an action of a group $G$ and a commuting algebraic action of a finite type algebraic group $H$. Suppose that $X$ is topologically $G \times H$ noetherian. Then it is also topologically $G$-noetherian.
\end{corollary}

\begin{example}
Let $X$ be a scheme that is topologically $G$-noetherian. Then any finite rank equivariant vector bundle over $X$ is also $G$-noetherian. For example, if $G$ is one of $\bO_{\infty}$, $\Sp_{\infty}$, or $\GL_{\infty}$, and $V$ is a finite length algebraic representation of $G$ (in the sense of \cite{infrank}) then $X=\Gr_r(V^*)$ is $G$-noetherian by \cite{eggsnow} (extending the main result of \cite{draisma}), and so any finite sum of finite tensor products of the rank $r$ tautological bundle and its dual is also $\bG$-noetherian.
\end{example}

\subsection{Overview of proof}

Let notation be as in Theorem~\ref{thm:Ghilbert}. The idea of the proof is as follows. First, we can proceed by noetherian induction on $S$, that is, we can assume that for every $G$-stable closed subset $S'$ of $S$ such that $S' \ne S$, the space $X_{S'}$ is topologically $G$-noetherian. This allows us to freely pass to open subsets of $S$. Second, finite type schemes over an arbitrary base behave similarly (at least in the ways that we care about) to finite type schemes over a field, assuming we are always allowed to replace the base with an open subset. Combining these two observations, we can in effect pretend that $S$ is the spectrum of a field, and then the result is obvious. The technical details of the actual proof are more complicated, but this is at least the intuition.

\subsection{Application}

We mention one (now defunct) application of our theorem. In \cite{ess}, we used Draisma's theorem mentioned above to prove a vast generalization of Stillman's conjecture: we showed that any invariant of ideals satisfying certain natural conditions is ``degreewise bounded'' (Stillman's conjecture being the case where the invariant is projective dimension). A preliminary version of \cite{ess}, written prior to Draisma's theorem, proved that Draisma's theorem was in fact equivalent to our generalization of Stillman's conjecture. Our proof that ``generalized Stillman'' implies Draisma's theorem required Corollary~\ref{cor:Ghilbert}, which is what propelled us to prove Theorem~\ref{thm:Ghilbert} in the first place.

\subsection*{Acknowledgments}

We thank Bhargav Bhatt for helpful conversations.

\section{Preliminaries from topology}

We omit proofs of the following, which are all standard exercises in topology.

\begin{proposition} \label{prop:dense-inverse}
Let $f \colon X \to S$ be an open map of topological spaces and let $U$ be a dense open subset of $S$. Then $f^{-1}(U)$ is a dense open subset of $X$.
\end{proposition}


\begin{proposition} \label{prop:dense-sub}
Let $X$ be a topological space, let $Z$ be a proper closed set, and let $U$ be a dense open subset of $X$. Then $Z \cap U$ is a proper closed subset of $U$.
\end{proposition}


\begin{proposition} \label{prop:noeth-pieces}
Let $X$ be a topological space, and let $X_1, \ldots, X_n$ be subspaces whose union is $X$. Suppose that each $X_i$ is noetherian. Then $X$ is noetherian.
\end{proposition}


\begin{remark}
Let $X$ be a topological space on which $G$ acts. We can then consider the quotient space $X/G$. The $G$-stable open (or closed, or irreducible) subsets of $X$ correspond to the open (or closed, or irreducible) subsets of $X/G$. Thus we can translate $G$-properties of $X$ to usual properties of $X/G$. For instance, $X$ is $G$-noetherian if and only if $X/G$ is noetherian. In this way, the above results can be applied in the $G$-noetherian setting.
\end{remark}

\section{Dimension vectors}

Let $\sD$ be the set of finite sequences $\delta=(\delta_1, \ldots, \delta_n)$, of variable length, where $\delta_1 \ge \delta_2 \ge \cdots$ and each $\delta_i$ is a non-negative integer. By convention, for $\delta \in \sD$ of length $n$, we put $\delta_i=-\infty$ for $i>n$. We order $\sD$ lexicographically, so that $\delta<\delta'$ if $\delta_1<\delta'_1$, or if $\delta_1=\delta_1'$ and $\delta_2<\delta_2'$, and so on. The unique minimal element of $\sD$ is the sequence $\delta$ of length~0; it has $\delta_i=-\infty$ for all $i$.

\begin{lemma}
  $(\sD,\le)$ is a well-order.
\end{lemma}

\begin{proof}
  It suffices to show that any strictly decreasing chain $\lambda^1 > \lambda^2 > \cdots$ must be finite. We prove this statement by induction on $\lambda^1_1$. If $\lambda^1_1=0$, this is clear as the chain can only consist of 1 element. So suppose $\lambda^1_1>0$ and that there is an infinite decreasing chain. Then the sequence of non-negative integers $\lambda^1_1 \ge \lambda^2_1 \ge \lambda^3_1 \ge \cdots$ is eventually constant, say equal to $c$; remove the finitely many $\lambda^i$ which do not have this first term and renumber the partitions $\lambda^1 > \lambda^2 > \cdots$. So there is a value $k \le \ell(\lambda^1)$ so that the sequence $(\lambda^n_i)_n$ is constant for $i<k$ and $(\lambda^n_k)_n$ converges to $c'<c$. In that case, remove the finitely many $\lambda^i$ such that $\lambda^i_k \ne c'$ and again renumber them $\lambda^1 > \lambda^2 > \cdots$. Define $\mu^i$ by removing the first $k-1$ parts from $\lambda^i$. By induction, $\mu^1 > \mu^2 > \cdots$ is finite, which is a contradiction.
\end{proof}

The following result allows for induction on elements of $\sD$:

\begin{proposition} \label{prop:dim-induct}
For each $\delta \in \sD$, let $\cP(\delta)$ be a boolean value. Suppose that $\cP(\delta')$ is true for all $\delta'<\delta$ implies $\cP(\delta)$ is true. Then $\cP(\delta)$ is true for all $\delta \in \sD$.
\end{proposition}

\begin{proof}
Let $S \subset \sD$ be the set of $\delta$ for which $\cP(\delta)$ is false. Since $\le$ is a well-order, if $S$ were non-empty then there would be a minimal element $\delta \in S$. But then $\cP(\delta')$ is true for all $\delta'<\delta$, so $\cP(\delta)$ would be true as well, a contradiction. So $S$ is empty.
\end{proof}

Let $X$ be a finite type scheme over a field $\bk$, and let $X_1, \ldots, X_n$ be the irreducible components of $X$, ordered by dimension (with $\dim(X_1)$ largest). We define the {\bf dimension vector} of $X$, denoted $\delta(X)$, to be the sequence $(\delta_1, \ldots, \delta_n)$, where $\delta_i=\dim(X_i)$. We regard it as an element of $\sD$. If $\bk$ is separably closed, $\delta$ is invariant under extension to a larger field. For a finite type morphism of schemes $X \to S$ and $s \in S$, we let $\delta_s(X)$ be $\delta(X_{\ol{s}})$, where $\ol{s}$ is a separably closed point at $s$. We say that $X \to S$ is {\bf $\delta$-constant} if $\delta_s(X)$ is independent of $s$, and then write $\delta(X)$ for the common value.

\section{Preliminaries from algebraic geometry}

\begin{proposition} \label{prop:descent}
Let $S=\Spec(A)$ be an affine scheme, and write $A=\bigcup_{i \in I} A_i$ (directed union) where each $A_i$ is finitely generated as a $\bZ$-algebra. Put $S_i=\Spec(A_i)$.
\begin{enumerate}[\indent \rm (a)]
\item Let $X \to S$ be a morphism of finite presentation. Then there exists $i \in I$ and a morphism of finite type $X_i \to S_i$ such that $X=(X_i)_S$.
\item Let $X_i$ and $Y_i$ be schemes of finite type over $S_i$, and let $\varphi \colon (X_i)_S \to (Y_i)_S$ be a morphism of schemes over $S$. Then there exists $j \ge i$ and a morphism $\varphi_j \colon (X_i)_{S_j} \to (Y_i)_{S_j}$ such that $\varphi$ is the base change of $\varphi_j$ to $S$.
\item Let $X_i \to S_i$ be a finite type morphism such that $(X_i)_S \to S$ is flat. Then there exists some $j \ge i$ in $I$ such that $(X_i)_{S_j} \to S_j$ is flat.
\item Let $X$ be a scheme of finite presentation over $S$ and let $Y$ be a closed subscheme of $X$ that is also of finite presentation over $S$. Then there exists $i \in I$, a finite type scheme $X_i$ over $S_i$, and a closed subscheme $Y_i$ of $X_i$ such that $Y \subset X$ is the base change of $Y_i \subset X_i$.
\end{enumerate}
\end{proposition}

\begin{proof}
  Parts (a) and (b) are parts of \cite[Tag 01ZM]{stacks}. Part (c) is a special case of \cite[Tag 05LY]{stacks}. Part (d) follows from \cite[Tag 0B8W]{stacks}.
\end{proof}

\begin{proposition} \label{prop:genflat}
Let $X \to S$ be a finite type morphism, with $S$ reduced. Then there is a dense open subset $U$ of $S$ such that $X_U \to X$ is flat of finite presentation.
\end{proposition}

\begin{proof}
This follows from a general version of generic flatness \cite[Tag 052B]{stacks}.
\end{proof}

\begin{proposition} \label{prop:delta-const-noeth}
Let $X \to S$ be a finite type morphism of noetherian schemes. Then there are open sets $U_1, \ldots, U_n$ of $S$ with dense union such that $X_{U_i} \to U_i$ is $\delta$-constant.
\end{proposition}

\begin{proof}
This proof follows \cite[Tag 055A]{stacks} closely (this proposition is really just a refinement of loc.\ cit.). Since $S$ is noetherian, we can replace it with an open dense subscheme in which no two irreducible components intersect. Thus $S$ is the disjoint union of its irreducible components, so we may just assume $S$ is irreducible. By \cite[Tag 0551]{stacks}, after replacing $S$ with a non-empty open subset, we can find a surjective finite \'etale morphism $S' \to S$ with $S'$ irreducible such that all irreducible components of $X'_{\eta}$ are geometrically irreducible, where $\eta$ is the generic point of $S'$ and $X'=X \times_S S'$. Since $S' \to S$ is open, we may as well replace $S$ with $S'$. We may further assume $S$ is integral, as $\delta$ is insensitive to nilpotents.

Let $X_{1,\eta}, \ldots, X_{n,\eta}$ be the irreducible components of $X_{\eta}$. These are all geometrically irreducible by our reductions. Let $X_i$ be the closure of the image of $X_{i,\eta}$ in $X$. After replacing $S$ with a non-empty open subset, we can assume $X$ is the union of the $X_i$ \cite[Tag 054Y]{stacks}. After shrinking $S$ again, we can assume that $X_{i,s}$ is geometrically irreducible for all $s \in S$ \cite[Tag 0559]{stacks}. After shrinking $S$ yet again, we can assume that each fiber of $X \to S$ has at least $n$ irreducible components \cite[Tag 0554]{stacks}. Since $X_s=X_{1,s} \cup \cdots \cup X_{n,s}$, it follows that for every $s \in S$, the fiber $X_s$ has exactly $n$ irreducible components, namely the $X_{i,s}$, and they are each geometrically irreducible. Finally, by \cite[Tag 05F6]{stacks}, we can find an open subset of $S$ such that the fibers of $X_i \to S$ have constant dimension for each $i$.
\end{proof}

\begin{proposition} \label{prop:delta-const}
Let $X \to S$ be a finite type morphism of non-empty reduced schemes. Then there is a non-empty open subset $U \subseteq S$ such that $X_U \to U$ is flat of finite presentation and $\delta$-constant.
\end{proposition}

\begin{proof}
By Proposition~\ref{prop:genflat}, after replacing $S$ with a dense open subscheme, we can assume $X \to S$ is flat of finite presentation. Replacing $S$ with some affine open, we can assume $S$ is affine. By Proposition~\ref{prop:descent}, we can find a finite type morphism $X' \to S'$ with $S'$ noetherian such that $X$ is the base change of $X'$ along a morphism $f \colon S \to S'$. We may as well replace $S'$ with the scheme-theoretic image of $f$, which is just the reduced subscheme structure on $\ol{f(S)}$ \cite[Tag 056B]{stacks}. In particular, $f$ has dense image. By Proposition~\ref{prop:delta-const-noeth}, there is a non-empty open subset $U'$ of $S'$ such that $X'_{U'} \to U'$ is $\delta$-constant. Since $f(S)$ is dense in $S'$ it must meet $U'$. Therefore $U=f^{-1}(U')$ is a non-empty open subset of $S$ such that $X_U \to U$ is $\delta$-constant (and still flat of finite presentation).
\end{proof}

\begin{proposition} \label{prop:open-flat}
Let $f \colon X \to S$ be a flat morphism of finite presentation, let $U$ be an open dense subset of $S$, and let $Y$ be a proper closed subset of $X$. Then $Y_U$ is a proper closed subset of $X_U$.
\end{proposition}

\begin{proof}
Since $f$ is flat of finite presentation, it is open \cite[Tag 01UA]{stacks}. Thus $f^{-1}(U)$ is a dense open subset of $X$ (Proposition~\ref{prop:dense-inverse}). It follows that $Y_U=f^{-1}(U) \cap Y$ is a proper closed subset of $X_U=f^{-1}(U)$ (Proposition~\ref{prop:dense-sub}).
\end{proof}

\begin{proposition} \label{prop:delta-const}
Let $X \to S$ be a flat finite type morphism of noetherian schemes. Assume there is a dense subset $A$ of $S$ such that $s \mapsto \delta_s(X)$ is constant for $s \in A$. Then there is an open dense subset $U$ of $S$ such that $X_U \to U$ is $\delta$-constant.
\end{proposition}

\begin{proof}
Applying Proposition~\ref{prop:delta-const-noeth}, there are open subsets $U_1, \ldots, U_n$ of $S$ such that $U=U_1 \cup \cdots \cup U_n$ is dense and $X_{U_i} \to U_i$ is $\delta$-constant. Since $A$ is dense, it meets each $U_i$, and so $\delta(X_{U_i})$ is independent of $i$. It follows that $X_U \to U$ is $\delta$-constant.
\end{proof}

\begin{proposition} \label{prop:delta-decrease-noeth}
Let $f \colon X \to S$ be a finite type morphism of reduced noetherian schemes that is flat and $\delta$-constant. Let $Y$ be a proper closed subscheme of $X$. Then there exists a non-empty open subset $U$ of $S$ such that $Y_U \to U$ is $\delta$-constant and $\delta(Y_U)<\delta(X_U)$.
\end{proposition}

\begin{proof}
By Proposition~\ref{prop:delta-const-noeth}, there are open subsets $V_1, \ldots, V_n$ of $S$ such that $Y_{V_i} \to V_i$ is $\delta$-constant and $V=V_1 \cup \cdots \cup V_n$ is dense in $S$. By Proposition~\ref{prop:open-flat}, $Y_V$ is a proper closed subset of $X_V$. Thus $Y_{V_i}$ is a proper subset of $X_{V_i}$ for some $i$. Put $U=V_i$ for this $i$. Thus $Y_U$ is a proper closed subset of $X_U$ and $Y_U \to U$ is $\delta$-constant. It remains to show that $\delta(Y_U)<\delta(X_U)$. Since both are $\delta$-constant, we can verify this over a generic point of $U$. 

So assume that $X$ is a reduced finite type scheme over a field and $Y$ is a closed subscheme. Let $X_1, \dots, X_d$ be the irreducible components of $X$ which have largest possible dimension $n$. Suppose one of them is not an irreducible component of $Y$. Then $\delta(X) = (n, n, \dots, n, \dots)$ with $n$ repeated $d$ times, but $\delta(Y)$ has $<d$ instances of $n$, so $\delta(Y) < \delta(X)$. In the other case, all of the $X_i$ are irreducible components of $Y$. Then both $\delta(Y)$ and $\delta(X)$ begin with $d$ instances of $n$, and we replace $X$ and $Y$ with the union $X'$ and $Y'$ of their components not equal to one of $X_1,\dots,X_d$. In particular, $X' \ne Y'$, and by induction on dimension, we have $\delta(X') > \delta(Y')$ which implies $\delta(X) > \delta(Y)$.
\end{proof}

\begin{proposition} \label{prop:delta-decrease}
  Let $f \colon X \to S$ be a flat morphism of finite presentation between reduced schemes that is $\delta$-constant. Let $Y$ be a proper closed subscheme of $X$. Then there is a non-empty open subset $U$ of $S$ such that $Y_U \to U$ is flat of finite presentation and $\delta$-constant with $\delta(Y_U)<\delta(X_U)$.

  Furthermore, if $X,S$ are $G$-schemes such that $f$ is $G$-equivariant and $Y$ and $U$ are $G$-stable, then we can take $U$ to be $G$-stable.
\end{proposition}

\begin{proof}
Since $Y$ is a closed subscheme of $X$ and $X$ is finite type over $S$, it follows that $Y$ is finite type over $S$. Thus, by Proposition~\ref{prop:genflat}, we can find an open dense subset $U$ of $S$ such that $Y_U \to U$ is flat of finite presentation. By Proposition~\ref{prop:open-flat}, $Y_U$ is a proper closed subset of $X_U$. Thus we may as well replace $S$ with $U$, and just assume that $Y$ is flat and of finite presentation over $S$.

Replace $S$ with an affine open so that $Y$ is still a proper closed subscheme of $X$. By Proposition~\ref{prop:descent}(d), there is a noetherian scheme $S'$, a finite type morphism $X' \to S'$, a closed subscheme $Y'$ of $X'$, and a morphism $g \colon S \to S'$ such that $Y \subset X$ is the pullback of $Y' \subset X'$. By Proposition~\ref{prop:descent}(c), we can assume that $X' \to S'$ is flat. We may as well replace $S'$ with the scheme-theoretic image of $g$, which is just $\ol{g(S)}$ with the reduced subscheme structure, so we can assume that $g(S)$ is dense in $S'$. Since $\delta_{g(s)}(X')=\delta_s(X)$ is constant for $s \in S$ and $g(S)$ is dense, it follows from Proposition~\ref{prop:delta-const} that there is a dense open subset $V'$ of $S'$ such that $X'_{V'} \to V'$ is $\delta$-constant. By Proposition~\ref{prop:open-flat}, $Y'_{V'}$ is a proper closed subset of $X'_{V'}$. By Proposition~\ref{prop:delta-decrease-noeth}, there is a non-empty open subset $U'$ of $V'$ such that $Y'_{U'} \to U'$ is $\delta$-constant with $\delta(Y'_{U'})<\delta(X'_{U'})$. Since $g(S)$ is dense in $S'$, it meets $U'$, and so $U=g^{-1}(U')$ is a non-empty open subset of $S$. Clearly, $Y_U \to U$ is $\delta$-constant with $\delta(Y)<\delta(X)$.

For the last statement, we use the above proof to get an open set $V$ and then take $U = \bigcup_{g \in G} gV$.
\end{proof}



\section{Proof of main results}

Consider the following statement, for $\delta \in \sD$.
\begin{quotation}
{\bf Statement $\cP(\delta)$.} Let $X \to S$ be a $G$-equivariant map of reduced schemes that is flat of finite presentation and $\delta$-constant with $\delta(X)=\delta$. Suppose that $S$ is topologically $G$-noetherian, and that for every proper $G$-stable closed subset $S'$ of $S$ the scheme $X_{S'}$ is topologically $G$-noetherian. Then $X$ is topologically $G$-noetherian.
\end{quotation}

\begin{lemma}
Statement $\cP(\delta)$ is true for all $\delta$.
\end{lemma}

\begin{proof}
  We proceed by induction on $\delta$ (Proposition~\ref{prop:dim-induct}). Thus let $\delta \in \sD$ and $X \to S$ as in Statement~$\cP(\delta)$ be given, and assume $\cP(\delta')$ holds for all $\delta'<\delta$. It suffices to show that every proper $G$-stable closed subset of $X$ is topologically $G$-noetherian. Thus let such a $Y$ be given, and endow $Y$ with the reduced subscheme structure. By Proposition~\ref{prop:genflat}, there is a non-empty open subset (which we may assume $G$-stable) $V$ so that $Y_V \to V$ is flat of finite presentation. By Proposition~\ref{prop:delta-decrease} there is a non-empty $G$-stable open subset $U$ of $V$ such that $Y_U \to U$ is flat of finite presentation and $\delta$-constant with $\delta'=\delta(Y_U)<\delta=\delta(X_U)$. Thus by $\cP(\delta')$, we have that $Y_U$ is $G$-noetherian. Since $X_{S \setminus U}$ is topologically $G$-noetherian, by the hypothesis of $\cP(\delta)$, the space $Y_{S \setminus U}$ is also topologically $G$-noetherian. It follows that $Y$ is topologically $G$-noetherian (Proposition~\ref{prop:noeth-pieces}), which completes the proof.
\end{proof}

\begin{proof}[Proof of Theorem~\ref{thm:Ghilbert}]
Let $X \to S$ be the given $G$-equivariant map of schemes. Since the statement is topological, we may assume that $X$ and $S$ are reduced. We proceed by noetherian induction on $S$: that is, we assume that for every proper closed subset $S'$ of $S$ the space $X_{S'}$ is topologically $G$-noetherian. By Proposition~\ref{prop:genflat}, Proposition~\ref{prop:delta-const}, and Proposition~\ref{prop:delta-decrease} there is a non-empty $G$-stable open subset $U$ of $S$ such that $X_U \to U$ is flat of finite presentation and $\delta$-constant. Put $\delta=\delta(X)$. By $\cP(\delta)$, it follows that $X_U$ is topologically $G$-noetherian. By the inductive hypothesis, $X_{S \setminus U}$ is topologically $G$-noetherian. Thus $X$ is topologically $G$-noetherian (Proposition~\ref{prop:noeth-pieces}).
\end{proof}

\begin{proof}[Proof of Corollary~\ref{cor:Ghilbert}]
Suppose that $X$ is topologically $G \times H$ noetherian, where $G$ is an arbitrary group and $H$ is a finite type algebraic group acting algebraically on $X$. Consider the action map $f \colon H \times X \to X$. Let $H$ act on $H \times X$ by $h \cdot (h', x)=(hh', x)$, and let $G$ act on $H \times X$ by $g \cdot (h,x)=(h,gx)$. Then $f$ is $G \times H$ equivariant. Since $H$ is finite type, so is $f$. The theorem therefore implies that $H \times X$ is topologically $G \times H$ noetherian. If $Z_{\bullet} \subset X$ is a descending chain of $G$-stable closed subsets then $H \times Z_{\bullet} \subset H \times X$ is a descending chain of $G \times H$ stable closed subsets, and thus stabilizes. Thus $Z_{\bullet}$ stabilizes, and so $X$ is topologically $G$-noetherian.
\end{proof}

\end{document}